\documentclass[12pt]{amsart}
\usepackage{amssymb,amsmath,amsthm,mathrsfs,multirow,xcolor,framed,url}
\usepackage[pdftex,
         pdfproducer={Latex with hyperref},
         pdfcreator={pdflatex}]{hyperref}
\newtheorem{theorem}{Theorem}[section]
\newtheorem{lemma}[theorem]{Lemma}

\newtheorem{drconj}[theorem]{Dyson's Rank Conjectures}

\theoremstyle{definition}

\newtheorem*{remark}{Remark}

\numberwithin{equation}{section}

\newcommand{\abs}[1]{\lvert#1\rvert}

\newcommand{\Parans}[1]{\left(#1\right)}

\newcommand\leg[2]{\genfrac{(}{)}{}{}{#1}{#2}} 

\newcommand\FL[1]{\left\lfloor#1\right\rfloor}

\allowdisplaybreaks
\newcommand\mylabel[1]{\label{#1}}

\newcommand\thm[1]{\ref{thm:#1}}
\newcommand\mythm[1]{\ref{thm:#1}}
\newcommand\lem[1]{\ref{lem:#1}}

\newcommand\myeqn[1]{(\ref{eq:#1})}

\newcommand{\beqs}{\begin{equation*}}
\newcommand{\eeqs}{\end{equation*}}
\newcommand{\beq}{\begin{equation}}
\newcommand{\eeq}{\end{equation}}
\renewcommand{\MR}[1]{\href{http://www.ams.org/mathscinet-getitem?mr={#1}}{MR{#1}}}

\newcommand{\Np}{N_\psi}

\newcommand\nutwid{\overset {\text{\lower 3pt\hbox{$\sim$}}}\nu}
\newcommand\sgtwid{\overset {\text{\lower 3pt\hbox{$\sim$}}}{sg}}
\newcommand\Nptwid{\overset {\text{\lower 3pt\hbox{$\sim$}}}{\Np}}
\newcommand\eptwid{\overset {\text{\lower 3pt\hbox{$\sim$}}}{\varepsilon}}

\newcommand{\sg}{\mbox{sg}}
\newcommand\twidit[1]{\overset {\text{\lower 3pt\hbox{$\sim$}}}#1}
\newcommand\dtwidit[1]{\overset {\text{\lower 6pt\hbox{$\sim$}}}#1}
\newcommand\Rtwid{\twidit{R}}

\begin{document}

\title{A new approach to the Dyson rank conjectures}

\author{F. G. Garvan}
\address{Department of Mathematics, University of Florida, Gainesville,
FL 32611-8105}
\email{fgarvan@ufl.edu}
\thanks{The author was supported in part by a grant from 
        the Simon's Foundation (\#318714).
A preliminary version of this paper was presented October 1, 2020 at the 
Number Theory Seminar, St. Petersburg State University and Euler International Mathematical Institute, Russia. It was also presented October 3, 2020 at
the Special Session on $q$-Series and Related Areas in Combinatorics and Number 
Theory, A.M.S.~Fall Eastern Sectional Meeting.}



\date{\today}                    

\begin{abstract}
In 1944 Dyson defined the rank of a partition as the largest part
minus the number of parts, and conjectured that the residue of
the rank mod $5$ divides the partitions of $5n+4$ into five equal
classes. This gave a combinatorial explanation of Ramanujan’s
famous partition congruence mod $5$. He made an analogous conjecture
for the rank mod $7$ and the partitions of $7n+5$. In 1954 Atkin and
Swinnerton-Dyer proved Dyson’s rank conjectures by constructing
several Lambert-series identities basically using the theory of
elliptic functions. In 2016 the author gave another proof using
the theory of weak harmonic Maass forms. In this paper we describe
a new and more elementary approach using Hecke-Rogers series.
\end{abstract}

\dedicatory{Dedicated to the memory of Richard Askey, a great friend and mentor}

\maketitle

\section{Some guesses in the theory of partitions}
\mylabel{sec:intro}

In 1944, Freeman Dyson \cite{Dy44}, as an undergraduate at Cambridge,
wrote an article with the title
of this section, in which he made a number of conjectures related to
Ramanujan's famous partition congruences. Let $p(n)$ denote the number
of partitions of $n$. Ramanujan discovered and later proved 
three beautiful congruences for the partition function, namely
\begin{align}
p(5n+4) &\equiv 0 \pmod{5}, \mylabel{eq:ramcong5}\\
p(7n+5) &\equiv 0 \pmod{7}, \mylabel{eq:ramcong7}\\
p(11n+6) &\equiv 0 \pmod{11}. \mylabel{eq:ramcong11}
\end{align}
Dyson went on to remark that although there at least four different
proofs of \myeqn{ramcong5} and \myeqn{ramcong7}, it would be
more satisfying to have a direct proof of \myeqn{ramcong5}.
By this, he supposed whether there was some natural way of dividing
the partitions of $5n + 4$ into five equally numerous classes.
He went on to define the \textit{rank} of a partition as the
largest part minus the number of parts, and conjectured that the
residue of the rank mod $5$ does the job of dividing the partitions
of $5n+4$ into five equal classes. He also conjectured that the residue
of the rank mod $7$ in a similar way divides the partitions of $7n+5$
into seven equal classes thus explaining \myeqn{ramcong7}. 

More explicitly, Dyson denoted by $N(m,n)$, the number of partitions
of $n$ with rank $m$, and let $N(m,t,n)$ denote the number
of partitions of $n$ with rank congruent to $m$ mod $t$.
We restate
\begin{drconj}[1944]
\mylabel{conj:DRC}
For all nonnegative integers $n$,
\begin{align}
N(0,5,5n+4) &= N(1,5,5n+4) = \cdots  = N(4,5,5n+4) = \tfrac{1}{5}p(5n + 4),
\mylabel{eq:Dysonconj5}\\
N(0,7,7n+5) &= N(1,7,7n+5) = \cdots  = N(6,7,7n+5) = \tfrac{1}{7}p(7n + 5).    
\mylabel{eq:Dysonconj7}
\end{align}
\end{drconj}
The corresponding conjecture with modulus $11$ is definitely false.
Towards the end of his article, Dyson conjectured that there is a hypothetical
statistic he dubbed the \textit{crank}, whose residue mod $11$ should divide
the partitions of $11n+6$ into eleven classes thus explaining \myeqn{ramcong11}.
The later discovery of the crank is another story \cite{An-Ga88}.

The Dyson Rank Conjectures \myeqn{Dysonconj5} and \myeqn{Dysonconj7}
were first proved by Atkin and Swinnerton-Dyer \cite{At-SwD} in 1954.
Atkin and Swinnerton-Dyer's proof involved proving many theta-function
and generalized Lambert  (or Appell-Lerch) series identities using basically 
elliptic
function techniques. Their method also involved finding identities
for the generating functions of $N(k,5,5n+r)$ for all the residues
$r=0$, $1$, $2$, $3$ and $4$.  
We quote from their paper \cite[p.84]{At-SwD}:
\begin{quote}
It is noteworthy that we have to obtain at the same time all the results
stated in these theorems--- we cannot simplify the working so as to 
merely to obtain Dyson's identities.
\end{quote}

Only recently have new methods for approaching the
Dyson Rank Conjectures been found. In 2017 Hickerson and Mortenson 
\cite{Hi-Mo2017} used their
theory of Appell-Lerch series to obtain results for the Dyson rank function
including the Dyson Rank Conjectures. In 2019 the author \cite{Ga19a}
showed how the theory of harmonic Maass forms can be used to prove
the Dyson Rank Conjectures and much more.
In this paper we describe a new and more elementary method for proving 
Dyson's Rank Conjectures.
The method only relies on identities for Hecke-Rogers series.   
We describe these series identities in the next section.
We are able to obtain \myeqn{Dysonconj5} for partitions of $5n+4$ without
having to deal with the partitions of $5n+r$ for the other residues $r=0$,
$1$, $2$ and $3$. 

\section{Hecke-Rogers series}                          
\mylabel{sec:HRS}
The main theorem of this section is Theorem \mythm{rankids}, which contains
four two-variable generalized Hecke-Rogers series identities for
the Dyson rank generating function. 
Only two of these identities
are needed to prove Dyson's mod $5$ rank conjecture \myeqn{Dysonconj5}.
Regarding the other two identities, we will use one identity to obtain
Ramanujan's $5$-dissection rank identity from the Lost Notebook
\cite[p.20]{Ra1988}. The last identity is connected 
with Dyson's mod $7$ rank conjecture.

One of these Hecke-Rogers identities is known. Two follow
from a general result of Bradley-Thrush \cite{BrTh2020}. Our proof of
the remaining identity uses results of Hickerson and Mortenson \cite{Hi-Mo14}
for Appell-Lerch series.

Following \cite[p.84]{An1984} a \textit{Hecke-Rogers series} has the form
$$
\sum_{(n,m)\in D} (\pm 1)^{f(n,m)} q^{Q(n,m) + L(n,m)}
$$
where $Q$ is an indefinite binary quadratic form, $L$ is a linear form
and $D$ is a subset of $\mathbb{Z}^2$ for which $Q(n,m)\ge0$.
L.~J.~Rogers \cite[p.323]{Ro1894} found 
\beq
\prod_{n=1}^\infty (1 - q^n)^2 = 
\sum_{n=-\infty}^\infty \sum_{m=-\FL{n/2}}^{\FL{n/2}}               
(-1)^{n+m} q^{ (n^2-3m^2)/2 + (n+m)/2}
\mylabel{eq:HReta2}
\eeq
The first systematic study of such series was done by Hecke \cite{He1925}
who independently obtained Rogers identity.
Identities of this type arose in Kac and Petersen's \cite{Ka-Pe1980} work
on character formulas for infinite dimensional Lie algebras and string
functions. 
Andrews \cite{An1984} showed how identities of this type can be
derived using his constant term method. 

We have the following two-variable generating function for the rank
\cite[Eq.(7.2),p.66]{Ga88b}:
Then
\begin{equation}
R(z;q) := \sum_{n=0}^\infty\sum_{m} N(m,n) z^m q^n 
          = \sum_{n=0}^\infty \frac{q^{n^2}}{(zq;q)_n(z^{-1}q;q)_n}
          = \frac{(1-z)}{(q)_\infty} \sum_{n=-\infty}^\infty 
\frac{(-1)^n q^{n(3n+1)/2}}{(1-zq^n)},
\mylabel{eq:NRid}
\end{equation}
where the last equality follows easily from \cite[Eq.(7.10),p.68]{Ga88b}.
Here we are using the usual $q$-notation:
\begin{align*}
(a)_n &= (a;q)_n := (1-a) (1 - aq) (1 - aq^2) \cdots (1 - aq^{n-1}),\\
(a)_\infty &=(a;q)_\infty = 
\lim_{n\to\infty} (a;q)_n = \prod_{n=1}^\infty (1-aq^{n-1}),
\end{align*}
provided $\abs{q}<1$,
and recalling that $N(m,n)$ is the number of partitions of $n$ with rank $m$.
We will often use the
Jacobi triple product identity \cite[Theorem 3.4, p.461]{An1974} for
the theta-function $j(z;q)$:
\beq
j(z;q) := (z;q)_\infty (z^{-1}q;q)_\infty (q;q)_\infty
= \sum_{n=-\infty}^\infty (-1)^n z^n q^{n(n-1)/2}.
\label{eq:jacp}
\eeq
Also we use the following notation for theta-type $q$-products:
\beq
J_{b,a}:=j(q^a;q^b),\quad\mbox{and}\quad J_b:=(q^b;q^b)_\infty.
\mylabel{eq:Jba}
\eeq            
We need the following general result of Bradley-Thrush \cite{BrTh2020}.
\begin{theorem}[{Bradley-Thrush \cite[Theorem 7.5]{BrTh2020}}]
\mylabel{thm:jbt}
Let $p$, $q$, $x$ and $y$ be non-zero complex numbers such that
$\abs{p}$, $\abs{q}<1$ and let $k$ be a positive integer such that
$\abs{pq^{-k^2}}<1$. Then
\beq
j(y;q) \sum_{n=-\infty}^\infty \frac{(-1)^n p^{n(n+1)/2} x^n}{1-yq^{kn}} =
\sum_{m=-\infty}^\infty \sum_{n=0}^\infty (-1)^{m+n}
q^{ (n-m)(n-m+1)/2 } j(q^{-kn} x^{-1};p) y^m.
\mylabel{eq:genjbtid}
\eeq
\end{theorem}

We also need some notation and results of Hickerson and Mortenson \cite{Hi-Mo14},
\cite{Hi-Mo2017}. First we give Hickerson and Mortenson's definitions
of their functions $f_{a,b,c}(x,y,q)$, $m(x,q,z)$, $g_{a,b,c}(x,y,q,z_1,z_0)$
and $g(x,q)$:
\begin{equation}
f_{a,b,c}(x,y,q):=\sum_{\substack{\sg (r)=\sg(s)}} \sg(r)(-1)^{r+s}x^ry^sq^{a\binom{r}{2}+brs+c\binom{s}{2}}.
\mylabel{eq:fabcdef}
\end{equation}
where $\sg (r):=1$ for $r\ge 0$ and $\sg(r):=-1$ for $r<0$.
\begin{equation}
m(x,q,z):=\frac{1}{j(z;q)}\sum_{r=-\infty}^{\infty}\frac{(-1)^rq^{\binom{r}{2}}z^r}{1-q^{r-1}xz}.
\label{eq:mxqzdef}
\end{equation}
\begin{align}
&g_{a,b,c}(x,y,q,z_1,z_0)\nonumber\\
&\quad :=\sum_{t=0}^{a-1}
     (-y)^tq^{c\binom{t}{2}} j(q^{bt}x;q^a)
  \,m\Parans{-q^{a\binom{b+1}{2}-c\binom{a+1}{2}-t(b^2-ac)}\frac{(-y)^a}{(-x)^b},
  q^{a(b^2-ac)},z_0}
\label{eq:mdef}\\
&\qquad +\sum_{t=0}^{c-1} (-x)^t q^{a\binom{t}{2}} j(q^{bt}y;q^c)
\,m\Parans{ -q^{c\binom{b+1}{2}-a\binom{c+1}{2}-t(b^2-ac)}\frac{(-x)^c}{(-y)^b},
         q^{c(b^2-ac)},z_1}.
\nonumber 
\end{align}
\begin{equation}
g(x,q):=x^{-1}\Parans{ -1+\sum_{n= 0}^{\infty}
\frac{q^{n^2}}
{(x;q)_{n+1}(q/x;q)_n}}.
\mylabel{eq:gdef}
\end{equation}
From \myeqn{NRid} we see that $g(x,q)$ is related to the rank function $R(z;q)$
by 
$$
g(x,q) = x^{-1}\Parans{-1 + \frac{1}{1-x} R(x;q)},
$$
and
\beq
R(z;q) = (1-z) \Parans{ 1 + z g(z,q)}.
\mylabel{eq:Rzqgzqid}
\eeq
The function $g(z,q)$ is related to the $m$-function by
\beq
g(z,q) = - z^{-2} m(z^{-3}q,q^3,z^2) - z^{-1} m(z^{-3}q^2,q^3,z^2).
\mylabel{eq:gzqmid}
\eeq
See \cite[Eq.(26a)]{Hi-Mo2017}.

We need 
\begin{theorem}[{Hickerson and Mortenson \cite[Theorem 1.6]{Hi-Mo14}}]
\mylabel{thm:HiMoThm1p6}
Let $n$ be a positive integer. For generic $x$, $y \in \mathbb{C}^{*}$
$$
f_{n,n+1,n}(x,y,q) = 
g_{n,n+1,n}(x,y,q,y^n/x^2,x^n/y^n).
$$
\end{theorem}
\begin{remark}
By generic, Hickerson and Mortenson mean values that do not cause singularities
at poles
of Appell-Lerch series or quotients of theta-functions.
Also $\mathbb{C}^{*}$ is the set of non-zero complex numbers.
\end{remark}
We need some well-known properties of Appell-Lerch series 
\cite[Proposition 3.1]{Hi-Mo2017}
\begin{align}
m(x,q,z) &= m(x,q,qz),
\mylabel{eq:mid1a}\\
m(x,q,z) &= x^{-1} m(x^{-1},q,z^{-1}),
\mylabel{eq:mid1b}\\
m(qx,q,z) &= 1- xm(x,q,z),
\mylabel{eq:mid1c}\\
m(x,q,z_1) - m(x,q,z_0)
&= \frac{ z_0 J_1^3 j(z_1/z_0;q) j(x z_0 z_1;q) }
       {j(z_0;q) j(z_1;q) j(x z_0;q) j(xz_1;q)}
\mylabel{eq:mid1d},
\end{align}
for generic $x$, $z$, $z_0$, $z_1 \in \mathbb{C}^{*}$.
The following is a well-known three term theta-function identity
\begin{align}
&j( d;q ) j( {{b}{c}^{-1}};q ) j( abc;q) j( ad;q ) -j( b;q ) 
j( {d}{c}^{-1};q ) j( acd;q ) j( ab;q ) 
\mylabel{eq:weier}\\
&\qquad +
{b}{c}^{-1}j( c;q ) j( abd;q ) j( ac;q ) j( {d}{b}^{-1};q ) = 0.
\nonumber
\end{align}
See \cite[Theorem 4.1]{BrTh2020}.

We will also need the following lemma.
\begin{lemma}
\mylabel{lem:jzqm}
\begin{align}
&  z j(z^{-1}q;q) m(q,q^3,z) + z j(z^{-2}q;q) m(z^3q,q^3,z^{-1}) 
\nonumber\\
& \quad + z^2 j(zq;q) m(q,q^3,z^{-1}) + z^2 j(z^{2}q;q) m(z^{-3}q,q^3,z)
\mylabel{eq:jzqmid}\\
&= - j(z^2;q) m(z^{-3}q,q^3,z^2) + z j(z^2;q) m(z^3q,q^3,z^{-2}).
\nonumber
\end{align}
\end{lemma}
\begin{proof}
We apply \myeqn{mid1d} three times to obtain
\begin{align*}
& z j(z^{-1}q;q) m(q,q^3,z) +  z^2 j(zq;q) m(q,q^3,z^{-1}) \\
&= z^2 j(zq;q) \left( m(q,q^3,z^{-1}) -  m(q,q^3,z)\right) \\
& = \frac{z^3 j(z^{-2};q^3) j(q;q^3) j(zq;q)}
         {j(z;q^3) j(z^{-1};q^3) j(zq;q^3) j(z^{-1}q;q^3)},
\end{align*}
\begin{align*}
& z j(z^{-2}q;q) m(z^3q,q^3,z^{-1}) - z j(z^2;q) m(z^3q,q^3,z^{-2})\\
&= z j(z^{2}q,q) \left(m(z^3q;q^3,z^{-1}) -  m(z^3q,q^3,z^{-2})\right)\\
& = \frac{j(z^{2};q^3) j(z;q^3) j(q;q^3)}
         {z j(z^{-2},q^3) j(z^{-1};q^3) j(zq;q^3) j(z^{2}q;q^3)},
\end{align*}
\begin{align*}
& j(z^2;q) m(z^{-3}q,q^3,z^2) + z^2 j(z^{2}q;q) m(z^{-3}q,q^3,z)\\
&= j(z^2,q) \left(m(z^{-3}q,q^3,z^2) - m(z^{-3}q,q^3,z)\right)\\
& = \frac{z j(z^{2};q) j(q,q^3)}
         {j(z^{2};q^3) j(z^{-2}q;q^3) j(z^{-1}q;q^3)}.
\end{align*}
Thus we see that \myeqn{jzqmid} is equivalent to showing that
\begin{align*}
& \frac{z^3 j(z^{-2};q^3) j(q;q^3) j(zq;q)}
         {j(z;q^3) j(z^{-1};q^3) j(zq;q^3) j(z^{-1}q;q^3)}\\
& + \frac{j(z^{2};q^3) j(z;q^3) j(q;q^3)}
         {z j(z^{-2};q^3) j(z^{-1};q^3) j(zq;q^3) j(z^{2}q;q^3)}\\
& + \frac{j(z^{2};q) j(q;q^3)}
         {j(z^{2};q^3) j(z^{-2}q;q^3) j(z^{-1}q;q^3)}=0.
\end{align*}
By rewriting each function on base $q^3$ we find that this identity
is equivalent to
$$
z j(zq^2;q^3) j(z^2q^2;q^3) j(z;q^3) -
j(zq^2;q^3) j(zq;q^3) j(z^2;q^3) + j(z;q^3) j(zq;q^3) j(z^2q;q^3) = 0.
$$
This identity is a special case of \myeqn{weier}. This completes the proof of
\myeqn{jzqmid}.
\end{proof}

\begin{theorem}
\mylabel{thm:rankids}
\begin{align}
&(zq)_\infty (z^{-1}q)_\infty (q)_\infty R(z;q) 
\mylabel{eq:rankid1}\\
&\quad = \frac{1}{2}\sum_{n=0}^\infty 
\left( \sum_{j=0}^{[n/2]} (-1)^{n+j}
(z^{n-3j}  + z^{3j-n}) q^{\frac{1}{2}(n^2-3j^2) + \frac{1}{2}(n-j)}\right.
\nonumber\\
& \hskip 2in  + \left. \sum_{j=1}^{[n/2]} (-1)^{n+j}
(z^{n-3j+1}  + z^{3j-n-1}) q^{\frac{1}{2}
(n^2-3j^2) + \frac{1}{2}(n+j)}\right),
\nonumber\\
&(zq)_\infty (z^{-1}q)_\infty (q)_\infty R(z;q^2) 
 = \sum_{n=0}^\infty 
 \sum_{j=-n}^{n} (-1)^{j}
    z^j q^{\frac{1}{2}n(3n+1) -j^2}(1- q^{2n+1})
\mylabel{eq:rankid2}\\
&(1+z)(z^2q;q)_\infty (z^{-2}q;q)_\infty (q;q)_\infty R(z;q) 
 = \sum_{n=0}^\infty 
   \sum_{j=-\lfloor n/2 \rfloor}^{\lfloor n/2\rfloor} (-1)^{n+j}
(z^{n+1}  + z^{-n}) q^{\frac{1}{2}(n^2-3j^2) + \frac{1}{2}(n-j)}
\mylabel{eq:rankid3}\\
&(1+z)(z^2q^2;q^2)_\infty (z^{-2}q^2;q^2)_\infty (q^2;q^2)_\infty R(z;q) 
\mylabel{eq:rankid4}\\
&\quad =
\sum_{n=0}^\infty 
   \sum_{j=0}^{2n} (-1)^{n}
(z^{j+1}  + z^{-j}) q^{3n^2+2n -\frac{1}{2}j(j+1)}
 - \sum_{n=1}^\infty 
   \sum_{j=0}^{2n-2} (-1)^{n}
(z^{j+1}  + z^{-j}) q^{3n^2-2n -\frac{1}{2}j(j+1)}
\nonumber
\end{align}
\end{theorem}
\begin{remark}
By letting $z=1$ we see that identities \myeqn{rankid1} and \myeqn{rankid3} are 
$z$-analogs  of \myeqn{HReta2}. Similarly we find that \myeqn{rankid2}
is $z$-analog of the identity \cite[Eq.(7.2),p.66]{An86}
$$
(q)_\infty^2 (q;q^2)_\infty = 
 \sum_{n=0}^\infty \sum_{j=-n}^{n} (-1)^{j}
    q^{\frac{1}{2}n(3n+1) -j^2}(1- q^{2n+1}).
$$
\end{remark}
\subsubsection*{Proof of \myeqn{rankid1}}
Equation \myeqn{rankid1} is \cite[Eq. (1.15), p.269]{Ga15a}.
For a proof see \cite[Section 3]{Ga15a}.

\subsubsection*{Proof of \myeqn{rankid2}}
We show how \myeqn{rankid2} follows from the following result
of Bradley-Thrush's Theorem \mythm{jbt}. In Equation \myeqn{genjbtid}
we let $k=2$, $p=q^6$, $x=q^{-2}$, $y=z$ and noting that 
$\abs{pq^{-k^2}}= \abs{q}^2 <1 $ we find that
$$
j(z;q) \sum_{n=-\infty}^\infty \frac{(-1)^n q^{n(3n+1)} x^n}{1-zq^{2n}} =
\sum_{m=-\infty}^\infty \sum_{n=0}^\infty (-1)^{m+n}
q^{ (n-m)(n-m+1)/2 } j(q^{-2(n-1)};q^6) z^m.
$$   
Now
$$
j(q^{-2(n-1)};q^6) = (-1)^{n+1}\leg{n-1}{3}q^{-n(n+1)/3} (q^2;q^2)_\infty,
$$
which follows easily from Jacobi's triple product \myeqn{jacp}.
See also \cite[Eq.(4.8)]{Ch-Ga2020a} or \cite[p.99]{At-SwD}. By this and 
\myeqn{NRid} we have
\begin{align*}
\frac{j(z;q)}{1-z}\, R(z;q^2) &= 
\sum_{m=-\infty}^\infty \sum_{n=0}^\infty (-1)^{m+1}
\leg{n-1}{3} q^{ m(m-1)/2 - mn + n(n+1)/6} z^m \\
&=\sum_{m=-\infty}^\infty \sum_{n=0}^\infty (-1)^{m+1}
\leg{n-1}{3} q^{ (m^2+\abs{m})/2 + n\abs{m} + n(n+1)/6} z^m.
\end{align*}
For the last equation we used symmetry in $z$.
Then in this last sum we replace $n$ by $n-3\abs{m}$ 
and use \myeqn{jacp} to obtain
\begin{align*}
(zq)_\infty (z^{-1}q)_\infty (q)_\infty R(z;q^2) 
&= 
\sum_{m=-\infty}^\infty \sum_{n \ge 3\abs{m}} (-1)^{m+1}
\leg{n-1}{3} q^{ n(n+1)/6 -m^2} z^m \\
&=\sum_{n=0}^\infty \sum_{m=-\FL{n/3}}^{\FL{n/3}} (-1)^{m+1}
\leg{n-1}{3} q^{ n(n+1)/6 -m^2} z^m.
\end{align*}
We find that the result \myeqn{rankid2} follows by 
replacing $n$ by $3n+k$ where
$k=-1$, $0$.

\subsubsection*{Proof of \myeqn{rankid3}}
We prove \myeqn{rankid3} by rewriting the right-side in terms  Hickerson and
Mortenson's $f_{1,2,1}$ and then by using one of their theorems and 
some identities for Appell-Lerch series.
From \myeqn{fabcdef} we find that
\begin{align*}
&f_{1,2,1}(z^{-1}q,z^{-2}q,q) \\
&= \sum_{n=0}^\infty \sum_{j=0}^{\FL{n/2}} (-1)^{n+j} z^{-n} 
q^{\frac{1}{2}(n-3j^2) + \frac{1}{2}(n-j)}
+
\sum_{n=0}^\infty \sum_{j=-\FL{n/2}}^{-1} (-1)^{n+j} z^{n+1} 
q^{\frac{1}{2}(n-3j^2) + \frac{1}{2}(n-j)}.
\end{align*}
Therefore
\beq
f_{1,2,1}(z^{-1}q,z^{-2}q,q)  + z f_{1,2,1}(zq,z^2q,q)\\
= \sum_{n=0}^\infty \sum_{j=-\FL{n/2}}^{\FL{n/2}} (-1)^{n+j} 
\left(z^{n+1} + z^{-n} \right)
q^{\frac{1}{2}(n-3j^2) + \frac{1}{2}(n-j)}.
\mylabel{eq:fHRid}
\eeq
From Theorem \thm{HiMoThm1p6} with $n=1$ we have
\begin{align}
&f_{1,2,1}(z^{-1}q,z^{-2}q,q)  + z f_{1,2,1}(zq,z^2q,q)
\mylabel{eq:fgjzmid}\\
&=g_{1,2,1}(z^{-1}q,z^{-2}q,q,z^{-1},z)  + z g_{1,2,1}(zq,z^2q,q,z,z^{-1})
\nonumber\\
&= j(z^{-1}q,q) m(q,q^3,z) + j(z^{-2}q,q) m(z^3q,q^3,z^{-1}) 
\nonumber\\
&\quad + z j(zq;q) m(q,q^3,z^{-1}) + zj(z^{2}q,q) m(z^{-3}q,q^3,z)
\nonumber\\
&=j(z^2;q) \left( m(z^3q,q^3,z^{-2}) - z^{-1} m(z^{-3}q,q^3,z^2)\right)
\nonumber
\end{align}
by Lemma \lem{jzqm}. By \myeqn{Rzqgzqid}, \myeqn{gzqmid},
\myeqn{mid1b}--\myeqn{mid1c}, \myeqn{fgjzmid} and \myeqn{fHRid} we have
\begin{align*}
&(1+z)(z^2q;q)_\infty (z^{-2}q;q)_\infty (q;q)_\infty R(z;q) 
= j(z^2;q)\,\frac{1}{1-z}\,R(z,q) \\
& = j(z^2;q)(1 + z\,g(z,q))\\
& = j(z^2;q)(1 - z^{-1} m(z^{-3}q,q^3,z^2) - m(z^{-3}q^2,q^3,z^2))\\
& = f_{1,2,1}(z^{-1}q,z^{-2}q,q)  + z f_{1,2,1}(zq,z^2q,q)\\
&= \sum_{n=0}^\infty \sum_{j=-\FL{n/2}}^{\FL{n/2}} (-1)^{n+j} 
\left(z^{n+1} + z^{-n} \right)
q^{\frac{1}{2}(n-3j^2) + \frac{1}{2}(n-j)}.
\end{align*}

\subsubsection*{Proof of \myeqn{rankid4}}
The proof of \myeqn{rankid4} is similar to that of \myeqn{rankid2}.
From \myeqn{NRid} we have
\begin{align*}
&(1+z)(z^2q^2;q^2)_\infty (z^{-2}q^2;q^2)_\infty (q^2;q^2)_\infty R(z;q) 
= \frac{j(z^2;q^2)}{1-z} R(z;q)\\
& = \frac{j(z^2;q^2)}{(q)_\infty} 
\sum_{n=-\infty}^\infty \frac{(-1)^n q^{\frac{1}{2}n(3n+1)}}{(1-zq^n)}
\\
& = \frac{j(z^2;q^2)}{(q)_\infty}
\sum_{n=-\infty}^\infty \frac{(-1)^n q^{\frac{1}{2}n(3n+1)}(1+zq^{n})}{(1-z^2q^{2n})}.
\end{align*}
Next we make two applications of Theorem \thm{jbt}, with $k=1$, $p=q^3$,
$q\mapsto q^2$, $y=z^2$ and $x=q^{-1}$ then $x=1$, so that
\begin{align*}
&(1+z)(z^2q^2;q^2)_\infty (z^{-2}q^2;q^2)_\infty (q^2;q^2)_\infty R(z;q) 
\\
&=\frac{1}{(q)_\infty}
\left(
\sum_{m=-\infty}^\infty \sum_{n=0}^\infty
(-1)^{m+n} q^{(n-m)(n-m+1)} j(q^{-2n+1};q^3) z^{2m} \right.\\
&\qquad \left. + 
\sum_{m=-\infty}^\infty \sum_{n=0}^\infty
(-1)^{m+n} q^{(n-m)(n-m+1)} j(q^{-2n};q^3) z^{2m+1}
\right).
\end{align*}
Now
$$
j(q^{n};q^3) = (-1)^{n+1}\leg{n}{3}q^{-\frac{1}{6}(n-1)(n-2)} (q)_\infty,
$$
which follows from Jacobi's triple product \myeqn{jacp}.
Then after some simplification and utilising symmetry in $z$ we find that
\begin{align*}
&(1+z)(z^2q^2;q^2)_\infty (z^{-2}q^2;q^2)_\infty (q^2;q^2)_\infty R(z;q) 
\\
&=
\sum_{n=0}^\infty \sum_{m=0}^{\FL{n/3}} (-1)^n \leg{n+1}{3}
q^{\frac{1}{3}n(n+2) - m(2m+1)}(z^{-2m} + z^{2m+1})\\
&\qquad +
\sum_{n=0}^\infty \sum_{m=0}^{\FL{n/3}} (-1)^n \leg{n}{3}
q^{\frac{1}{3}(n+1)(n+5) - m(2m+3)}(z^{-2m-1} + z^{2m+2}).
\end{align*}
Then after further simplification and series manipulation we obtain
the final result \myeqn{rankid4}. We omit these details.
This completes the proof of Theorem \thm{rankids}.

\section{Proof of Dyson's rank conjecture mod $5$}
\mylabel{sec:proofDRC5}
\subsection{$5$-dissections of some theta-products}                          
\mylabel{subsec:5disstheta}

Let $p$ be a positive integer and $F(q)$ be a power series in $q$
$$
F(q) = \sum_n a(n) q^n.
$$
The \textit{$p$-dissection} of $F(q)$ splits a series into $p$ parts
according to the residue mod $p$ of the exponent of $q$ and is given by
\beq
F(q) = \sum_{r=0}^{p-1} \sum_{n\equiv r\pmod{p}} a(n) q^n = 
\sum_{r=0}^{p-1} q^r F_r(q^p).
\mylabel{eq:pdissdef}
\eeq
The Atkin $U$-operators pick out a part of this $p$-dissection
\beq
U_{p,r}(F(q)) := F_r(q) = \sum_n a(pn+r) q^n
\mylabel{eq:Uprdef}
\eeq
for $0 \le r \le p-1$.
We also define the operators $U_{p,m}^{*}$, and $A_{p,m}$ by
\begin{align*}
U_{p,m}^{*} (F(q)) &= \sum_n a(pn + m) q^{pn+m},   \\
A_{p,m}(F) &= \sum_n a(n) q^{(n-m)/p}, 
\end{align*}
so that
$$
U_{p,m} = A_{p,0} \circ U_{p,m}^{*}.
$$
The following $5$-dissections are well-known and can be proved
with little more than Jacobi's triple product identity \myeqn{jacp}.

\begin{lemma}
\mylabel{lem:5disstheta}
Let $\zeta= \exp(2\pi i/5)$. Then
\begin{align}
(\zeta q)_\infty (\zeta^{-1} q)_\infty (q)_\infty &=
J_{25,10} + q (\zeta^2 + \zeta^{-2}) J_{25,5},
\mylabel{eq:5diss1}\\
E(q) &:= (q)_\infty = 
J_{25}\left(
{\frac {J_{{25,10}}}{J_{{25,5}}}} -q -q^2{\frac {J_{{25,5}}}{J_{{25,10}}}}
\right),
\mylabel{eq:5diss2}\\
\theta_4(q) &= \sum_{n=-\infty}^\infty (-1)^n q^{n^2} 
= J_{50,25} - 2q J_{50,15} + 2q^4 J_{50,5}
\mylabel{eq:5diss3}
\end{align}
\end{lemma}
\begin{proof}
Equation \myeqn{5diss1} follows from \cite[Lemma 3.19]{Ga88b}. See 
\cite[Lemma 3.18]{Ga88b} for equation \myeqn{5diss2}. The general $p$-dissection
of $E(q)$, where $(p,6)=1$, is due to Atkin and Swinnerton-Dyer 
\cite[Lemma 6]{At-SwD}. The proof depends on Euler's Pentagonal
Number Theorem \cite[p.11]{An-book},
 and Watson's quintuple product identity \cite[Lemma 5]{At-SwD}.
Equation \myeqn{5diss3} follows easily from Jacobi's triple product
identity \myeqn{jacp}.
\end{proof}

The following results follow easily from Lemma \lem{5disstheta}.
\begin{lemma}
\mylabel{lem:5sifttheta}
\begin{align}
U_{5,2}\left( E(q)^2 \right) &= - J_5^2,
\mylabel{eq:sift1}\\
U_{5,3}\left(\theta_4(q)\, E(q) \right) 
&= 2\,\frac{J_5 J_{5,1} J_{10,3}}
          {J_{5,2}},
\mylabel{eq:sift2}\\
U_{5,4}\left(\theta_4(q)\, E(q) \right) 
&= 2\,\frac{J_5 J_{5,2} J_{10,1}}
          {J_{5,1}}.
\mylabel{eq:sift3}
\end{align}
\end{lemma}

\subsection{$5$-dissections involving the rank function and Hecke-Rogers series}
\mylabel{subsec:5dissrankHR}
We begin by
letting $z=1$ in \myeqn{rankid1} to obtain the identity
\beq
(q)_\infty^2 =  E(q)^2 = 
\sum_{n=0}^\infty 
 \sum_{j=-\lfloor n/2\rfloor}^{\lfloor n/2 \rfloor} (-1)^{n+j}
     q^{\frac{1}{2}(n^2-3j^2) + \frac{1}{2}(n-j)},
\mylabel{eq:HRE12}
\eeq
which, as mentioned before,  
is originally due to L.~J.~Rogers \cite[p.323]{Ro1894}.

\begin{lemma}
\mylabel{lem:zR5dis}
Let $\zeta=\exp(2\pi i/5)$. Then
$$
U_{5,2}\left( 
(\zeta q)_\infty (\zeta^{-1} q)_\infty (q)_\infty \,R(\zeta,q)
\right)
= - J_5^2.
$$
\end{lemma}
\begin{proof}
We have
$$
\frac{1}{2}(n^2-3j^2) + \frac{1}{2}(n-j) \equiv 2\pmod{5}
$$
if and only if $n\equiv 2 \pmod{5}$ and $j\equiv 4 \pmod{5}$,
in which case $n-3j\equiv0\pmod{5}$. Similarly 
$$
\frac{1}{2}(n^2-3j^2) + \frac{1}{2}(n+j) \equiv 2\pmod{5}
$$
if and only if $n\equiv 2 \pmod{5}$ and $j\equiv 1 \pmod{5}$,
in which case $n-3j+1\equiv0\pmod{5}$. Thus 
from \myeqn{rankid1}, \myeqn{HRE12} and \myeqn{sift1} we have
\begin{align*}
&U_{5,2}\left( 
(\zeta q)_\infty (\zeta^{-1} q)_\infty (q)_\infty \,R(\zeta,q)
\right)\\
&=U_{5,2}
\left(\sum_{n=0}^\infty 
 \sum_{j=-\lfloor n/2\rfloor}^{\lfloor n/2 \rfloor} (-1)^{n+j}
     q^{\frac{1}{2}(n^2-3j^2) + \frac{1}{2}(n-j)}\right)\\
&=U_{5,2}\left( E(q)^2 \right) = -J_5^2.
\end{align*}
\end{proof}
Next we use the Hecke-Rogers identity \myeqn{rankid2}.
By letting $z=1$ we see that this identity is a $z$-analog of
\beq             
\theta_4(q) \, E(q)
=
\frac{(q)_\infty^3}{(q^2;q^2)_\infty} 
= \sum_{n=0}^\infty 
 \sum_{j=-n}^{n} (-1)^{j}
 q^{\frac{1}{2}n(3n+1) -j^2}(1- q^{2n+1}).
\mylabel{eq:T4EHRid}
\eeq       
\begin{lemma}
\mylabel{lem:zR52dis}
Let $\zeta=\exp(2\pi i/5)$. Then
\begin{align}
U_{5,3}\left( 
(\zeta q)_\infty (\zeta^{-1} q)_\infty (q)_\infty \,R(\zeta,q^2)
\right)
&= (\zeta^2 + \zeta^3)\,\frac{J_5 J_{5,1} J_{10,3}}
          {J_{5,2}},
\mylabel{eq:U53R2id}
\\
U_{5,4}\left( 
(\zeta q)_\infty (\zeta^{-1} q)_\infty (q)_\infty \,R(\zeta,q^2)
\right)
&= (\zeta + \zeta^4)\,\frac{J_5 J_{5,2} J_{10,1}}
          {J_{5,1}}.
\mylabel{eq:U54R2id}
\end{align}
\end{lemma}
\begin{proof}
We have 
$$
\frac{1}{2}n(3n+1) -j^2 \equiv 3\pmod{5}
$$
if and only if $n\equiv 1 \pmod{5}$ and $j\equiv \pm 2 \pmod{5}$,
or 
$n\equiv 2 \pmod{5}$ and $j\equiv \pm 2 \pmod{5}$.  
Similarly 
we have 
$$
\frac{1}{2}n(3n+1) +(2n+1) -j^2 \equiv 3\pmod{5}
$$
if and only if $n\equiv 2 \pmod{5}$ and $j\equiv \pm 2 \pmod{5}$,
or 
$n\equiv 3 \pmod{5}$ and $j\equiv \pm 2 \pmod{5}$.  
Thus from \myeqn{rankid2} we have
\begin{align*}
&U_{5,3}^{*}\left( 
(\zeta q)_\infty (\zeta^{-1} q)_\infty (q)_\infty \,R(\zeta,q^2)
\right)
=U_{5,3}^{*}
\left( \sum_{n=0}^\infty 
 \sum_{j=-n}^{n} (-1)^{j}
\zeta^j q^{\frac{1}{2}n(3n+1) -j^2}(1- q^{2n+1})
\right) \\
&= \zeta^2\,
\left( 
\sum_{\substack{n\ge0\\ n\equiv 1,2\,\pmod{5}}}
\sum_{\substack{-n \le j \le n\\ j\equiv 2\, \pmod{5}}}
(-1)^j 
q^{\frac{1}{2}n(3n+1) -j^2}\right.\\
&\qquad\qquad \left.
-
\sum_{\substack{n\ge0\\ n\equiv 2,3 \pmod{5}}}
\sum_{\substack{-n \le j \le n\\ j\equiv 2 \pmod{5}}}
(-1)^j 
q^{\frac{1}{2}n(3n+1) + 2(n+1)-j^2}
\right) \\
&\quad +
\zeta^3\,
\left( 
\sum_{\substack{n\ge0\\ n\equiv 1,2\,\pmod{5}}}
\sum_{\substack{-n \le j \le n\\ j\equiv 3\, \pmod{5}}}
(-1)^j 
q^{\frac{1}{2}n(3n+1) -j^2}\right.\\
&\qquad\qquad \left.
-
\sum_{\substack{n\ge0\\ n\equiv 2,3 \pmod{5}}}
\sum_{\substack{-n \le j \le n\\ j\equiv 3 \pmod{5}}}
(-1)^j 
q^{\frac{1}{2}n(3n+1) + 2(n+1)-j^2}
\right) \\
&= \frac{1}{2}(\zeta^2 + \zeta^3)\,
\left( 
\sum_{\substack{n\ge0\\ n\equiv 1,2\,\pmod{5}}}
\sum_{\substack{-n \le j \le n\\ j\equiv 2,3 \pmod{5}}}
(-1)^j 
q^{\frac{1}{2}n(3n+1) -j^2}\right.\\
&\qquad\qquad \left.
-
\sum_{\substack{n\ge0\\ n\equiv 2,3 \pmod{5}}}
\sum_{\substack{-n \le j \le n\\ j\equiv 2,3 \pmod{5}}}
(-1)^j 
q^{\frac{1}{2}n(3n+1) + 2(n+1)-j^2}
\right) \\
&=\frac{1}{2}(\zeta^2 + \zeta^3) 
 U_{5,3}^{*}
\left( \sum_{n=0}^\infty 
 \sum_{j=-n}^{n} (-1)^{j}
 q^{\frac{1}{2}n(3n+1) -j^2}(1- q^{2n+1})
\right).       
\end{align*}
Thus by \myeqn{T4EHRid} and \myeqn{sift2} we have
\begin{align*}
&U_{5,3}\left( 
(\zeta q)_\infty (\zeta^{-1} q)_\infty (q)_\infty \,R(\zeta,q^2)
\right)
=\frac{1}{2}(\zeta^2 + \zeta^3) U_{5,3}((q)_\infty^3\,R(1,q^2))\\
&=\frac{1}{2}(\zeta^2 + \zeta^3) U_{5,3}(\theta_4(q) \,E(q))     
= (\zeta^2+\zeta^3)\,\frac{J_5 J_{5,2} J_{10,1}}
          {J_{5,1}},
\end{align*}
which is equation \myeqn{U53R2id}.
The proof of \myeqn{U54R2id} is analogous.
\end{proof}
\subsection{Completing the proof of Dyson's mod $5$ rank conjecture}         
\mylabel{subsec:pfDRC5}
We start by letting $z=\zeta=\exp(2\pi i/5)$ in the generating function
for the rank. We have
\begin{align*}
R(\zeta,q) &= \sum_{n=0}^\infty \sum_m N(m,n) \zeta^m q^n 
           = \sum_{n=0}^\infty \sum_{r=0}^4 
\sum_{\substack{m \equiv r\\\mbox{\scriptsize mod $5$}}}  
N(m,n) \zeta^m q^n \\
           &= \sum_{n=0}^\infty \Parans{\sum_{r=0}^4 N(r,5,n) \zeta^r} q^n,
\end{align*}
and
$$
\mbox{Coefficient of $q^{5n+4}$ in} R(\zeta,q)
=
\sum_{r=0}^4 N(r,5,5n+4) \zeta^r.                
$$
Therefore Dyson's mod $5$ rank conjecture \myeqn{Dysonconj5} is equivalent
to showing
\beq
U_{5,4}(R(\zeta,q))=0.
\mylabel{eq:DRC5a}
\eeq
See \cite[Lemma 2.2]{Ga88b}.
Since we have identities on both base $q$ and $q^2$ we instead prove
the equivalent identity
\beq
U_{5,3}(R(\zeta,q^2))=0,
\mylabel{eq:DRC5b}
\eeq
and write the $5$-dissection of $R(\zeta,q^2)$:
$$
R(\zeta,q^2) = \sum_{k=0}^4 q^k R_k(q^5).
$$
We obtain three linear equations for the functions
$R_2(q)$, $R_3(q)$ and $R_4(q)$.
In Lemma \lem{zR5dis} we replace $q$ by $q^2$ and use \myeqn{5diss1}
to find
\beq
(\zeta^2+\zeta^3) J_{10,2} \, R_2(q) + J_{10,4}\,R_4(q) =-J_{10}^2.
\mylabel{eq:EQ1}
\eeq
By equation \myeqn{5diss1} and Lemma \lem{zR52dis} we have
the following two equations.
\begin{align}
(\zeta^2+\zeta^3) J_{5,1} \,R_2(q) + J_{5,2} \,R_3(q) &=
(\zeta^2+\zeta^3) \frac{ J_{5,1} J_5 J_{10,3}}{J_{5,2}},
\mylabel{eq:EQ2}\\
(\zeta^2+\zeta^3) J_{5,1} \,R_3(q) + J_{5,2} \,R_4(q) &=
(\zeta+\zeta^4) \frac{ J_{5,2} J_5 J_{10,1}}{J_{5,1}}.
\mylabel{eq:EQ3}
\end{align}
Solving equations \myeqn{EQ1}--\myeqn{EQ3} we find that
\begin{align}
R_3(q) &= 
\frac{1}{D(q)} \left(
J_{5}\,J_{5,1}^3\,J_{10,2}\,J_{10,3}\,J_{10,4}
-2\,J_{5}\,J_{5,1}^2\,J_{5,2}\,J_{10,1}\,J_{10,4}^2
+J_{5}\,J_{5,2}^3\,J_{10,1}\,J_{10,2}\,J_{10,4} \right.
\mylabel{eq:EQNSOL3}\\
&\quad +J_{10}^2\,J_{5,1}^3\,J_{5,2}\,J_{10,4}
-J_{10}^2\,J_{5,1}\,J_{5,2}^3\,J_{10,2}
-J_{5,2}\,(J_{5}\,J_{5,1}^2\,J_{10,1}\,J_{10,4}^2
\nonumber\\
&\qquad 
\left.
+J_{5}\,J_{5,1}\,J_{5,2}\,J_{10,2}^2\,J_{10,3}
-J_{5}\,J_{5,2}^2\,J_{10,1}\,J_{10,2}\,J_{10,4}
-J_{10}^2\,J_{5,1}^3\,J_{10,4})\,(\zeta^2+\zeta^3)
\right),
\nonumber
\end{align}
where
\begin{align*}
D(q) &=
J_{5,1}^4\,J_{10,4}^2
+J_{5,1}^2\,J_{5,2}^2\,J_{10,2}\,J_{10,4}
-J_{5,2}^4\,J_{10,2}^2\\
&=
1-6\,q+10\,{q}^{2}+4\,{q}^{3}-19\,{q}^{4}-10\,{q}^{6}+64\,{q}^{7}-9\,
{q}^{8}-66\,{q}^{9}-40\,{q}^{11}+\cdots.
\end{align*}
We have
\beq
\frac{J_{10,1}\,J_{10,4}}{J_{10}^2}
=
\frac{J_{5,1}}{J_{5}},\qquad
\frac{J_{10,2}\,J_{10,3}}{J_{10}^2}
=
\frac{J_{5,3}}{J_{5}}.
\mylabel{eq:Jsimp}
\eeq
This together with \myeqn{EQNSOL3} and some simplification we have
$$
R_3(q) = 0,
$$
which completes the proof of \myeqn{DRC5b} and thus Dyson's mod $5$
rank conjecture \myeqn{Dysonconj5}.

It can be shown that $D(q)$
is an eta-quotient
$$
D(q) = \frac{J_{10}^3 J_{1}^6}{J_{5}^2 J_2} = \frac{1}{q}
\frac{\eta^3(10\tau) \eta^6(\tau)}{\eta^2(5\tau) \eta(2\tau)}.
$$
This can be proved using standard modular function techniques, but this
identity is not needed in the proof of \myeqn{DRC5b}.

Since $R_3(q)=0$, equations \myeqn{EQ2} and \myeqn{EQ3} imply
that $R_2(q)$ and $R_4(q)$ have product forms:
\beq
R_2(q) = \frac{J_{10}^2}{J_{10,2}},\quad
R_4(q) = -(1 + \zeta^2 + \zeta^3)\,\frac{J_{10}^2}{J_{10,4}}.
\mylabel{eq:R2+4prodforms}
\eeq

\section{Proof of Ramanujan's mod $5$ rank identity}
Again let $\zeta=\exp(2\pi i/5)$. The following identity appears on p.20 of 
Ramanujan's Lost Notebook \cite{Ra1988}.
\begin{align}
R(\zeta,q) &= A(q^5) + (\zeta + \zeta^{-1}-2)\,\phi(q^5) + q\,B(q^5) + 
(\zeta+\zeta^{-1})\,q^2\,C(q^5)
\label{eq:Ramid5}\\
&\quad 
- (\zeta+\zeta^{-1})\,q^3\left\{D(q^5) - (\zeta^2 + \zeta^{-2}  - 2)\frac{\psi(q^5)}{q^5}
\right\},
\nonumber
\end{align}
where
$$
A(q) = \frac{(q^2,q^3,q^5;q^5)_\infty}{(q,q^4;q^5)_\infty^2},\,
B(q) = \frac{(q^5;q^5)_\infty}{(q,q^4;q^5)_\infty},\,
C(q) = \frac{(q^5;q^5)_\infty}{(q^2,q^3;q^5)_\infty},\,
D(q) = \frac{(q,q^4,q^5;q^5)_\infty}{(q^2,q^3;q^5)_\infty^2},
$$
and
$$
\phi(q) = -1 + \sum_{n=0}^\infty \frac{q^{5n^2}}{(q;q^5)_{n+1} (q^4;q^5)_n},\qquad
\psi(q) =  -1 + 
\sum_{n=0}^\infty \frac{q^{5n^2}}{(q^2;q^5)_{n+1} (q^3;q^5)_n}.
$$
We write the $5$-dissection of $R(\zeta,q)$:
$$
R(\zeta,q) = \sum_{k=0}^4 q^k R_k(q^5).
$$
From equations \myeqn{DRC5a} and \myeqn{R2+4prodforms} we have
\beq
R_1(q) = \frac{J_5^2}{J_{5,1}}, \quad 
R_2(q) = (\zeta+\zeta^{-1})\,\frac{J_5^2}{J_{5,2}} \quad 
R_4(q)=0.
\mylabel{eq:R124dis5}
\eeq
It suffices to show that
\begin{align}
R_0(q) &= 
{\frac {{J_{{5}}}^{2}J_{{5,2}}}{{J_{{5,1}}}^{2}}}+ \left( {\zeta}^{4}+
\zeta-2 \right) \, \phi(q), 
\mylabel{eq:R05id}\\
R_3(q) &= 
- \left( {\zeta}^{4}+\zeta \right)\,
\frac{J_{{5,1}}{J_{{5}}}^{2}}{{J_{{5,2}}}^{2}}+
\frac{1}{q}\,\left( 2\,{\zeta}^{3}+2\,{\zeta}^{2}+1\right)\,\psi(q).
\mylabel{eq:R35id}
\end{align}
This time we use the Hecke-Rogers identity \myeqn{rankid3}.
By letting $z=1$ we see that this is a different $z$-analog of 
\myeqn{HRE12}.
We define
\beq
\Rtwid(z,q) = (1+z)(z^2q;q)_\infty (z^{-2}q;q)_\infty (q;q)_\infty R(z,q) .
\mylabel{eq:Rtwiddef}
\eeq
\begin{lemma}
\mylabel{lem:U503ids}
Let $\zeta= \exp(2\pi i/5)$. Then
\begin{align}
U_{5,0}\left(\Rtwid(\zeta,q)\right) &=
(1+ \zeta)\, \frac{J_5^2\,J_{5,2}^2}{J_{5,1}^2}
- (2 + 2\zeta + \zeta^3) \,\left(\Rtwid(q,q^5) - J_{5,2}\right),
\mylabel{eq:U50}\\
U_{5,4}\left(\Rtwid(\zeta,q)\right) &= 
(\zeta^2 + \zeta^4)\, \frac{J_5^2\,J_{5,1}^2}{J_{5,2}^2}
- \frac{1}{q}\,(2 + 2\zeta + \zeta^3) \,\left(\Rtwid(q^2,q^5) - J_{5,1}\right).
\mylabel{eq:U54}
\end{align}
\end{lemma}
\begin{proof}
We have
$$
\frac{1}{2}(n^2-3j^2) + \frac{1}{2}(n-j) \equiv 0\pmod{5}
$$
if and only if $(n,j) \equiv (0,0)$, $(0,3)$, $(1,4)$, $(3,4)$, $(4,0)$
or $(4,3)\pmod{5}$. We have the following table
$$
\begin{array}{ccc}
\noalign{\hrule}
\zeta^{n+1} + \zeta^{-n} & n & j \\
\noalign{\hrule}
1+\zeta & 0  & 0\\
1+\zeta & 0  & 3\\
-{\zeta}^{3}-\zeta-1 &  1 &  4\\
-{\zeta}^{3}-\zeta-1 &  3 &  4\\
1+\zeta              &  4 &  0\\
1+\zeta              &  4 &  3\\
\noalign{\hrule}
\end{array}
$$
We let
$$
\mathcal{S}_1 = \{(0,0),\,  (0,3),\,  (1,4),\,  (3,4),\,  (4,0),\, (4,3)\},\qquad
\mathcal{S}_2 = \{(1,4),\,  (3,4)\},
$$
so that by \myeqn{rankid3} we have
\begin{align}
&U_{5,0}^*\left(\Rtwid(\zeta,q)\right)  
\mylabel{eq:U50sums}\\
&\qquad = 
(1+\zeta) \underset{(n,j)\in \mathcal{S}_1 \pmod{5}}
{\sum_{n=0}^\infty \sum_{j=-\lfloor n/2 \rfloor}^{\lfloor n/2\rfloor}}
(-1)^{n+j}
q^{\frac{1}{2}(n^2-3j^2) + \frac{1}{2}(n-j)} 
\nonumber\\
&\qquad\qquad\qquad -(2 + 2\zeta + \zeta^3) 
\underset{(n,j)\in \mathcal{S}_2 \pmod{5}}
{\sum_{n=0}^\infty \sum_{j=-\lfloor n/2 \rfloor}^{\lfloor n/2\rfloor}}
(-1)^{n+j}
q^{\frac{1}{2}(n^2-3j^2) + \frac{1}{2}(n-j)}
\nonumber
\end{align}
Also by \myeqn{rankid3} we have
\beq
\Rtwid(q,q^5) = 
{\sum_{n=0}^\infty \sum_{j=-\lfloor n/2 \rfloor}^{\lfloor n/2\rfloor}}
(-1)^{n+j}\left(q^{n+1} + q^{-n}\right)
q^{\frac{5}{2}(n^2-3j^2) + \frac{5}{2}(n-j)}.
\mylabel{eq:Rtwidq15id}
\eeq
Now we let
\beq
V(n,j) = 
\frac{1}{2}(n^2-3j^2) + \frac{1}{2}(n-j),
\mylabel{eq:Vdef}
\eeq
and find that
\begin{align}
\frac{1}{5} V(5n+1,-5j-1) &= 5\, V(n,j) - n,
\nonumber                
\\
\frac{1}{5} V(5n+3,-5j-1) &= 5\, V(n,j) +n+ 1.
\mylabel{eq:V513}
\end{align}
We solve some inequalities.
$$
-\lfloor \tfrac12(5n+1) \rfloor \le -5j-1 \le \lfloor \tfrac12(5n+1) \rfloor  
\Longleftrightarrow
\begin{cases}
-m \le j \le m-1 & \mbox{if $n=2m$},\\
-m \le j \le m & \mbox{if $n=2m+1$},
\end{cases}
$$
and
$$
-\lfloor \tfrac12(5n+3) \rfloor \le -5j-1 \le \lfloor \tfrac12(5n+3) \rfloor  
\Longleftrightarrow
\begin{cases}
-m \le j \le m & \mbox{if $n=2m$},\\
-m-1 \le j \le m & \mbox{if $n=2m+1$}.
\end{cases}
$$
It follows that
\begin{align}
& A_{5,0}\left(
\underset{(n,j)\in \mathcal{S}_2 \pmod{5}}
{\sum_{n=0}^\infty \sum_{j=-\lfloor n/2 \rfloor}^{\lfloor n/2\rfloor}}
(-1)^{n+j}
q^{\frac{1}{2}(n^2-3j^2) + \frac{1}{2}(n-j)}\right)
\mylabel{eq:A50}\\
&
= 
{\sum_{n=0}^\infty \sum_{j=-\lfloor n/2 \rfloor}^{\lfloor n/2\rfloor}}
(-1)^{n+j}\left(q^{n+1} + q^{-n}\right)
q^{\frac{5}{2}(n^2-3j^2) + \frac{5}{2}(n-j)} 
- \sum_{m=-\infty}^\infty (-1)^m q^{m(5m+1)/2}
\nonumber\\
&
=
\Rtwid(q,q^5) - J_{5,2},
\nonumber
\end{align}
by  \myeqn{rankid3}, \myeqn{Rtwiddef} and Jacobi's
triple product identity \myeqn{jacp}.
Now
\begin{align}
&\underset{(n,j)\in \mathcal{S}_1 \pmod{5}}
{\sum_{n=0}^\infty \sum_{j=-\lfloor n/2 \rfloor}^{\lfloor n/2\rfloor}}
(-1)^{n+j}
q^{\frac{1}{2}(n^2-3j^2) + \frac{1}{2}(n-j)} 
\mylabel{eq:bigsum50}\\
&=
U_{5,0}^{*}\left(
{\sum_{n=0}^\infty \sum_{j=-\lfloor n/2 \rfloor}^{\lfloor n/2\rfloor}}
(-1)^{n+j}
q^{\frac{1}{2}(n^2-3j^2) + \frac{1}{2}(n-j)} 
\right)
\nonumber\\
&=
U_{5,0}^{*}\left(
(q)_\infty^2
\right) \qquad\qquad \mbox{(by \myeqn{HRE12})}
\nonumber\\
&= 
\frac{J_{25}^2\,J_{25,10}^2}{J_{25,5}^2},
\nonumber
\end{align}
by \myeqn{5diss2}. Thus from \myeqn{U50sums}, \myeqn{A50} and \myeqn{bigsum50}
we have
\beq
U_{5,0}\left(\Rtwid(\zeta,q)\right)  
 =(1+ \zeta)\, \frac{J_5^2\,J_{5,2}^2}{J_{5,1}^2}
 - (2 + 2\zeta + \zeta^3) \,\left(\Rtwid(q,q^5) - J_{5,2}\right),
 \mylabel{eq:U50id}
\eeq
which is \myeqn{U50}. The proof of \myeqn{U54} is analogous.
\end{proof}
Now by \myeqn{Rtwiddef} and the definition of $\phi(q)$ we have
\begin{align}
\Rtwid(q,q^5) 
& = (q^2;q^5)_\infty (q^3;q^5)_\infty (q^5;q^5)_\infty \frac{1}{1-q} R(q,q^5)
\mylabel{eq:Rtwidq15}\\ 
& = J_{5,2} \sum_{n=0}^\infty \frac{q^{5n^2}}{(q;q^5)_{n+1} (q^4;q^5)_n}
\qquad\qquad\mbox{(by Jacobi's triple product \myeqn{jacp})}
\nonumber \\
& = J_{5,2}\,\left( 1 + \phi(q)\right).
\nonumber
\end{align}

By replacing $z$ by $\zeta$ in \myeqn{Rtwiddef},
$\zeta$ by $\zeta^2$ in \myeqn{5diss1}
and by \myeqn{R124dis5} we have
\begin{align}
&\Rtwid(\zeta,q) 
\mylabel{eq:Rtwidzqdis}\\
&= (1+\zeta)\left(J_{25,10} + q (\zeta + \zeta^4) J_{25,5}\right)
   \left( R_0(q^5) + q\frac{J_{25}^2}{J_{25,5}} 
   + q^2(\zeta + \zeta^4) \frac{J_{25}^2}{J_{25,10}} 
   + q^3 R_3(q^5)\right)
\nonumber   
\end{align}
We apply $U_{5,0}$ to both sides of \myeqn{Rtwidzqdis} to find that
\beq
U_{5,0}\left(\Rtwid(\zeta,q)\right) = (1+\zeta) \, J_{5,2} \, R_0(q).
\mylabel{eq:U50left}
\eeq
By \myeqn{U50left} and \myeqn{U50id} we have
$$
(1+\zeta) \, J_{5,2} \, R_0(q) = (1 +\zeta) \frac{J_5^2 J_{5,2}^2}{J_{5,1}^2}
- (2 + \zeta + \zeta^3) \, J_{5,2} \,\phi(q),
$$
and we easily deduce \myeqn{R05id}. The proof of \myeqn{R35id} is similar.
\section{Equations for the rank mod $7$}
\mylabel{sec:rankmod7}
In this section we consider Dyson's mod $7$ rank conjecture 
\myeqn{Dysonconj7}. 
The following identity is the corresponding analog of \myeqn{5diss1}.
Through this section we assume $\zeta=\exp(2\pi i/7)$.
\beq
(\zeta q)_\infty (\zeta^{-1} q)_\infty (q)_\infty =
J_{49,21} + q (\zeta^2 + \zeta^{3} + \zeta^4 +\zeta^5) J_{49,14}
- q^3 (\zeta^3 + \zeta^4) J_{49,7}.
\mylabel{eq:zth7dis}
\eeq
The proof of the following lemma is analogous to the proof of Lemma \lem{zR52dis}
and depends on the Hecke-Rogers identities \myeqn{rankid1}, \myeqn{rankid3}
and \myeqn{rankid4}.
\begin{lemma}
\mylabel{lem:zR7dis1}
Let $\zeta= \exp(2\pi i/7)$. Then
\begin{align*}
U_{7,4}\left( 
(\zeta q)_\infty (\zeta^{-1} q)_\infty (q)_\infty \,R(\zeta,q)
\right)
&=  J_7^2, \\
U_{7,4}\left( 
(1+\zeta)(\zeta^2 q)_\infty (\zeta^{-2} q)_\infty (q)_\infty \,R(\zeta,q)
\right)
&=  2 \zeta^4 \, J_7^2\\
U_{7,3}\left( 
(1+\zeta)(\zeta^2 q^2;q^2)_\infty (\zeta^{-2} q^2;q^2)_\infty (q^2;q^2)_\infty 
\,R(\zeta,q)
\right)
&=  2 \zeta^4 \, \frac{J_{14}^3}{J_7}.
\end{align*}
\end{lemma}
We write the $7$-dissection of $R(\zeta,q)$:
$$
R(\zeta,q) = \sum_{k=0}^6 q^k R_k(q^7).
$$
Dyson's mod $7$ rank conjecture is equivalent to showing that
$$
U_{7,5}\left(R(\zeta,q)\right) = R_5(q) = 0.
$$
Unfortunately we have only been able to find three linear equations
for the functions
$R_1(q)$, $R_3(q)$ and $R_4(q)$.
\beq
-(\zeta^4 + \zeta^3)\,J_{7,1}\,R_1(q) 
+(\zeta^5 + \zeta^4 + \zeta^3 + \zeta^2)\, J_{7,2} \,R_3(q) 
+ J_{7,3}\,R_4(q) = J_7^2,
\mylabel{eq:EQN71}
\eeq
\beq
(\zeta^5 + \zeta^4 + \zeta^3)\,J_{7,1}\,R_1(q) 
+\zeta^4 \, J_{7,2} \,R_3(q) 
+ (\zeta+1)\,J_{7,3}\,R_4(q) = 2\,\zeta^4\,J_7^2,
\mylabel{eq:EQN72}
\eeq
\beq
\zeta^4 \,J_{14,4}\,R_1(q) 
+(\zeta+1) \, J_{14,6} \,R_3(q) 
+ (\zeta^5+\zeta^4+\zeta^3)\,q\,J_{14,2}\,R_4(q) 
= 2\,\zeta^4\,q\,\frac{J_{14}^3}{J_7}. 
\mylabel{eq:EQN73}
\eeq
Using these equations it is possible to show that
$$
R_1(q) = \frac{J_7^2}{J_{7,1}},\quad
R_3(q) = (\zeta^5+\zeta^2+1)\,\frac{J_7^2}{J_{7,2}},\quad
R_4(q) = -(\zeta^5+\zeta^2)\,\frac{J_7^2}{J_{7,3}}.         
$$
We have been unable to find 
a fourth linear equation only involving $R_1(q)$, $R_3(q)$, $R_4(q)$, $R_5(q)$.
This should be compared with equations \myeqn{EQ1}, \myeqn{EQ2} and
\myeqn{EQ3}, which were enough to prove the mod $5$ conjecture \myeqn{Dysonconj5}.

\section{Conclusion}
\mylabel{sec:end}

In this paper we presented a new approach to proving Dyson's rank conjectures. 
This new approach involved utilising various Hecke-Rogers identities.
We showed how this method gave a new proof for Dyson's mod 5 rank
conjecture as well as the related identity in Ramanujan's Lost Notebook.
We end by listing some problems.
\begin{enumerate}
\item[1.]
Extend the methods of this paper to prove Dyson's mod $7$ rank conjecture
\myeqn{Dysonconj7}, and find a new proof for the mod $7$ 
analog of
Ramanujan's identity \myeqn{Ramid5}. 
See \cite[p.16]{An-Be-RLNIII}.
\item[2.]
Find simple proofs of the four Hecke-Rogers identities 
\myeqn{rankid1}--\myeqn{rankid4}. Combinatorial proofs are also needed.
These results would lead to a truly elementary proof of Dyson's rank
conjectures.  
\item[3.]
Apply the methods of this paper to other rank-type functions, including
Lovejoy's overpartition rank \cite{Lo2005}, Berkovich and the author's
$M_2$-rank \cite{Be-Ga02}, and Jenning-Shaffer's exotic
Bailey-Slater spt-functions \cite{JeSh2017}.
\end{enumerate}

\subsection*{Data Availability Statement}
Data sharing not applicable to this article as the research of this paper
does not involve the use of any datasets.

\subsection*{Acknowledgments}
The author would like to thank
Jonathan Bradley-Thrush for showing how his Theorem \thm{jbt} can be used 
to prove the Hecke-Rogers identities \myeqn{rankid2} and \myeqn{rankid4}.

\end{document}